\documentclass[11pt]{article}

\title{\vskip-1.0em \sc A twisted inclusion between tensor products of operator spaces}

\author{\sc Yemon Choi}

\def\update{20th June 2016; updated 20th May 2020}

\date{\update}

\usepackage[dvipsnames]{xcolor}
\newcommand{\comment}[1]{{\sf\textbf{\textcolor{Blue}{[\small#1]}}}}
\renewcommand{\comment}[1]{}

\usepackage[a4paper,left=30mm,right=30mm,top=30mm,bottom=30mm]{geometry}
\usepackage{amsmath,amssymb}

\usepackage{sectsty}
\allsectionsfont{\sffamily}

\renewcommand{\emph}[1]{{\sl#1}\/}
\newcommand{\defeq}{\mathbin{:=}}
\newcommand{\dt}[1]{{\itshape#1}\/}

\newcommand{\para}[1]{\paragraph{#1.}} 

\newcommand{\norm}[1]{\Vert{#1}\Vert}
\newcommand{\Norm}[1]{\bigl\Vert{#1}\bigr\Vert}

\newcommand{\basetp}{{\otimes}}
\newcommand{\tp}{\mathbin{\basetp}}
\newcommand{\ptp}{\mathbin{\widehat{\basetp}}}
\newcommand{\mintp}{\mathbin{\basetp}_{\rm min}}
\newcommand{\htp}{\mathbin{\basetp}_h}

\newcommand{\ip}[2]{\langle#1,#2\rangle}

\newcommand{\COL}{\textsf{COL}}

\newcommand{\Kbar}{\overline{K}}

\newcommand{\cO}{{\mathcal O}}

\newcommand{\mbb}[1]{{\mathbb#1}}

\newcommand{\Cplx}{\mbb{C}}

\newcommand{\Mat}{\mbb{M}}

\newcommand{\CB}{\operatorname{\it CB}} 

\newcommand{\pcat}[1]{\textbf{\textsf{#1}}}
\newcommand{\OpSp}{\pcat{OpSp}}

\newcommand{\til}[1]{\widetilde{#1}}

\newcommand{\longopp}[1]{{#1}^{\sim}} 

\newcommand{\Cst}{\operatorname{C}^*}

\newenvironment{YCnum}{%
\begin{enumerate}

}{\end{enumerate}\ignorespacesafterend}

\RequirePackage[PS,nohug]{diagrams}

\newarrow{Eq} ===== 
\newarrow{Dash} {}{dash}{}{dash}>
\newarrow{Dots} ....>

\RequirePackage{amsthm}   

\newcounter{pulse}[section]
\numberwithin{pulse}{section}  

\newcommand{\thf}{\sc} 

\theoremstyle{plain}
\newtheorem{thm}[pulse]{\thf Theorem}

\newtheorem{prop}[pulse]{\thf Proposition}
\newtheorem{lem}[pulse]{\thf Lemma}

\theoremstyle{definition}

\theoremstyle{remark}
\newtheorem{rem}[pulse]{\thf Remark}

\begin{document}

\maketitle
\section*{Author comments added 20th May 2020}

{\sf
The original version of this document had several small mis-statements, all arising from an incorrect definition of the transpose operator (due to my careless blurring of the distinction between a Hilbert space and its complex conjugate). These have now been corrected, as were a few minor typos.

In the years since this note was originally written, I have decided that the application of the main result to the construction of $2$-cocycles on Fourier algebras deserves to be written up separately, in fuller detail.
Moreover, reflection during the last few years leads me to believe that Theorem~\ref{t:mainthm} should be part of a bigger picture that deserves a more thorough treatement, where the key interpolation trick is done at the level of operator spaces and not just at the level of Banach spaces.

Therefore I have removed the material on cocycles, and am leaving this note here as a historical record. \textbf{Currently I do not intend to submit this note for publication}, although special cases of some of the remarks and results will be used in the forthcoming paper on $2$-cocycles.

\hfill --- Y. Choi, Lancaster, 2020/5/20
}

\bigskip\hrule\newpage
\begin{abstract}
Given operator spaces $V$ and $W$, let $\til{W}$ denote the opposite operator space structure on the same underlying Banach space. Although the identity map $W\to \til{W}$ is in general not completely bounded, we show that the identity map on $V\tp W$ extends to a contractive linear map $V\ptp W \to V\mintp \til{W}$,
where $\ptp$ and $\mintp$ denote the projective and injective tensor products of operator spaces. In future work, this will be applied to construct antisymmetric $2$-cocycles on certain Fourier algebras.


\bigskip
\noindent\textit{Dedicated to John Rainwater, with thanks for his varied contributions and generous spirit.}
\end{abstract}

\begin{section}{Introduction}
We start by emphasising a convention that will be adhered to throughout this note: \emph{all operator spaces are assumed, as part of the definition, to be complete.}

A core result in the theory of operator spaces is the following observation: the operation on $K(\ell_2)$ given by matrix transpose, $a \mapsto a^\top$,  fails to be completely bounded, even though it is an isometric involution of Banach spaces. This fact serves to explain certain phenomena in non-commutative harmonic analysis, and can be exploited to prove structural results about Fourier and Fourier-Stieltjes algebras of locally compact groups: see, for instance, \cite{ForRun,RunUyg}.

More generally: given an operator space $W$ --- by which we mean a complex vector space~$W$, equipped with a sequence of \emph{complete} norms $\norm{\cdot}_{(n)}$ on $M_n(W) = \Mat_n \tp W$ that satisfy Ruan's axioms --- one may define a new sequence of norms as follows:
\[ \norm{ \sum_i a_i \tp w_i}_{(n), {\rm opp}} \defeq \norm{ \sum_i a_i^\top \tp w_i}_{(n)} \qquad(a_i\in \Mat_n, w_i\in W).\]
These matrix norms also define an operator space structure on $W$, which we denote by $\til{W}$ and call\footnotemark\ the \dt{opposite operator space} of $W$, or the \dt{opposite operator space structure} on the underlying Banach space of $W$.
\footnotetext{A more natural name might be the \dt{transposed operator space structure}, but the terminology here appears to be the standard one.}
 While $W$ and $\til{W}$ have the same underlying Banach space, in general they are not isomorphic as operator spaces; in particular, the identity map on $W$ is not completely bounded as a map $W\to\til{W}$.

All this is well known. However, the following result appears to be new. We write $\ptp$ for the projective tensor product of operator spaces and $\mintp$ for the injective tensor product (also known as the \dt{minimal} tensor product).

\begin{thm}[Main theorem]\label{t:mainthm}
Let $V$ and $W$ be operator spaces. The identity map on their algebraic tensor product $V\tp W$ extends to a linear contraction
\[ \Psi_{V,W}: V\ptp W \to V\mintp  \til{W}. \]
\end{thm}

Note that in general, $\Psi_{V,W}$ cannot be completely bounded (just take $V=\Cplx$), and so it seems hard to obtain a more direct proof by considering operator-space tensor norms on $V\tp W$.
We also note that in this theorem, we cannot replace the projective tensor product by the Haagerup tensor product, even if we weaken ``linear contraction'' to ``bounded linear map'' (Proposition~\ref{p:not-htp}).

Theorem~\ref{t:mainthm} was originally motivated by a technical issue that arose in studying the Hoch\-schild cohomology of Fourier algebras, specifically the problem of \dt{higher-dimensional weak amenability} as defined in \cite{BEJ_ndimWA}.
The main issue is that when $A$ and $B$ are cb-versions of Banach algebras, and $V$ and $W$ are cb-modules over $A$ and $B$ respectively, then $V\ptp\til{W}$ might not be a Banach module over $A\ptp B$; however, $V\mintp W$ \emph{is} such a module, and then Theorem~\ref{t:mainthm} allows us to replace the bad space $V\ptp\til{W}$ with the better space $V\mintp W$.
%
%
Details of this construction will be given in forthcoming work.

\para{Remarks on notation}
What we have written here as $\til{W}$ is often denoted in the literature by $W^{\rm op}$. We have chosen different notation because in some of the intended applications, one is dealing with an operator space $A$ which is also an algebra; and hence there is a potential conflict with the usage of $A^{\rm op}$ to denote the ``opposite algebra'', i.e. the algebra with the same underlying vector space but with reversed product.
 In longer expressions, when considering the opposite operator space, we use the notation $\longopp{(\dots)}$; for instance $\longopp{B(H)}$ denotes $B(H)$ equipped with the opposite of its usual operator space structure.

\end{section}

\begin{section}{Conventions and technical preliminaries}
We assume the reader is familiar with the basic definitions of operator spaces and completely bounded maps, as presented in \cite{ER_OSbook} or \cite{Pis_OSbook}. We remind her that, given operator spaces $E$ and $F$, the space $\CB(E,F)$ is itself an operator space in a natural way. If  $V$ is an operator space,  then the dual Banach space $V^*$ becomes an operator space under the identification $V^*=\CB(V,\Cplx)$, while $\CB(\Cplx,V)$ is completely isometrically isomorphic to~$V$.

Although we will not use any category theory, it is convenient occasionally to refer to the category of operator spaces and completely bounded maps, which we denote by \OpSp.

We shall abbreviate the phrase ``operator space structure'' to o.s.s.
Whenever $H$ is a Hilbert space and we refer to $B(H)$ as an operator space, we assume (unless explicitly stated otherwise) that it is equipped with its usual, canonical o.s.s.; note that if we do this, then there is a natural and completely isometric identification of $B(H)$ with $\CB(\COL_H)$, where $\COL_H$ denotes $H$ equipped with the column o.s.s.

\para{Opposite operator spaces, adjoints and transposes}
The \dt{opposite o.s.s.} was already defined in the introduction; see also \cite[\S2.10]{Pis_OSbook}. Let us collect some basic properties that do not seem to be mentioned in \cite{ER_OSbook} or \cite{Pis_OSbook}. It is easily checked that if $f:X \to Y$ is completely bounded, then so is $f: \til{X} \to \til{Y}$, with the same cb-norm. For sake of clarity, and to emphasise the functorial behaviour, we write this as $\til{f}: \til{X}\to\til{Y}$.
The same calculation gives, with some book-keeping, a more precise result: we omit the details.

\begin{lem}\label{l:functorial-opp}
Given operator spaces $E$ and $F$, we have a completely isometric isomorphism $\longopp{\CB(E,F)} \cong_1 \CB(\til{E},\til{F})$. In particular, we can identify $\longopp{(E^*)}$ with $(\til{E})^*$.
\end{lem}

Let $K$ be a Hilbert space. If $b\in B(K)$ then we denote its Banach space adjoint by $b^\#: K^*\to K^*$.
It is usual to identify $K^*$ with the conjugate Hilbert space $K$; this has the same underlying set as $K$, but
 scalar multiplication is now defined by $\lambda\cdot\xi= \overline{\lambda}\xi$.
The inner product on $\Kbar$ is defined by
\[ \ip{\xi}{\eta}_{\Kbar} \defeq \ip{\eta}{\xi}_H \]
When we make this identfication, $b^\#$ is identified with the operator $b^\top\in B(\Kbar)$ that satisfies
\[ \ip{b^\top\xi}{\eta}_{\Kbar}
= \ip{b\eta}{\xi}_K
= \ip{\xi}{b\eta}_{\Kbar} \]
This defines the \dt{transpose operator} $\top:B(K) \to B(\Kbar)$ (note that this is linear, not conjugate-linear).
From the remarks above (or by direct calculation) we see that $\top$ is a complete isometry from $B(K)$ onto $\longopp{B(\Kbar)}$, and
hence from $\longopp{B(K)}$ onto $B(\Kbar)$.

\para{Tensor products and tensor norms}
For clarity, we repeat some notation. The algebraic tensor product of two complex vector spaces $E$ and $F$ is denoted by $E\tp F$. 
The projective and injective tensor products in the category \OpSp\ are denoted by $\ptp$ and $\mintp$ respectively; this follows the notation of \cite{Pis_OSbook}, rather than that of \cite{ER_OSbook}.
If $f\in \CB(E,X)$ and $g\in \CB(F,Y)$ then by tensoring we obtain completely bounded maps $E\ptp F\to X\ptp Y$ and $E\mintp F\to X\mintp Y$; for extra emphasis, these maps will be denoted by $f\ptp g$ and $f\mintp g$ respectively. We have $\norm{f\ptp g}_{cb}=\norm{f\mintp g}_{cb}=\norm{f}_{cb}\norm{g}_{cb}$.

In proving Theorem~\ref{t:mainthm}, we exploit the fact that the injective tensor norm on $B(H)\tp B(K)$ can be calculated in terms of the action of this algebra on $S_2(\Kbar,H)$, the space of Hilbert-Schmidt operators $\Kbar\to H$.
It is slightly more convenient to replace $K$ by $\Kbar$, and to state this well-known result in a slightly different form from the usual one in the literature.

\begin{lem}\label{l:HS-trick}\
\begin{YCnum}
\item
The linear map $\theta_0: B(H) \tp B(\overline{K}) \to B(S_2(K,H))$ that is defined by
$\theta_0( a\tp b^\top) (c) \defeq acb$
extends to an injective $*$-homomorphism $\theta: B(H)\mintp B(\Kbar) \to B(S_2(K,H))$.

\item
There is a complete isometry $\Lambda: B(H) \mintp \longopp{B(K)} \to B(S_2(K,H))$, which satisfies $\Lambda(a\tp b)(c)=acb$ for all $a\in B(H)$, $b\in B(K)$ and $c\in S_2(K,H))$. 
\end{YCnum}
\end{lem}

\begin{proof}[Sketch of the proof]
Part (i) can be proved by representing a given Hilbert-Schmidt operator as an element of the Hilbert-space tensor product $H\otimes_2 \Kbar$, and checking the required $*$-homomorphism property by direct calculations;
 see also Proposition 2.9.1 in \cite{Pis_OSbook}, or the calculations in Section 3.5 of~\cite{ER_OSbook}.
Part (ii) follows from part (i) by composing $\theta$ with the completely isometric isomorphism
 $\iota \tp \top: B(H) \mintp \longopp{B(K)} \to B(H) \mintp B(\overline{K})$.
\end{proof}

Instead of the ``concrete'' definition of the matrix norms on the projective tensor product of operator spaces, we prefer to use its characteristic universal property: it linearizes those bilinear maps $h:E\times F \to G$ which are ``completely bounded'' in the following sense
\[ \sup_{m,n\geq 1} \left\{ \norm{ h(x_{ij}, y_{pq}) }_{M_{mn}(E\tp F)} \colon [x_{ij}]\in {\sf ball}_1M_m(E)\;,\; [y_{pq}]\in {\sf ball}_1M_n(F) \right\} < \infty. \]
(Here, our terminology is that of \cite{ER_OSbook}; the reader should beware that often such maps are instead called ``jointly completely bounded'', and that the term ``completely bounded'' is then used for what \cite{ER_OSbook} call ``multiplicatively bounded''.) \comment{The CMB1991 paper of E+R calls the (j)cb maps ``matricially bounded''.}
Denoting the space of such maps by $\CB_{\rm bil}(E\times F; G)$, there are natural and completely isometric identifications
\[  \CB(E\ptp F, G) \cong_1 \CB_{\rm bil}(E\times F;G) \cong_1 \CB( E, \CB(F,G)). \]
(See \cite[Proposition 7.1.2]{ER_OSbook}.)
 We note, for future reference, that there is a natural and completely isometric identification of $\til{E}\ptp \til{F}$ with $\longopp{(E\ptp F)}$.

\begin{rem}\label{r:cc-actions}
If $H$ is a Hilbert space equipped with column o.s.s., the natural left action of $B(H)$ on $H$ defines a completely contractive bilinear map $B(H) \times H \to H$. This can be verified directly, but it is more illuminating to see it as a special case of the following general fact: given operator spaces $E$, $F$ and $G$, the composition operation
\[ \CB(F,G) \times \CB(E,F) \to \CB(E,G) \]
is completely contractive as a bilinear map. 
 (To deduce the original statement, take $E=\Cplx$ and $F=G=\COL_H$.) Similarly,
the adjoint action of $B(H)$ on $H^*$ defines a completely contractive bilinear map $H^*\times B(H) \to H^*$,
given by $(\phi,b) \mapsto b^\#\phi$.
 (Take $E=F=\COL_H$ and $G=\Cplx$.) Note that here $H^*$ is equipped with the dual o.s.s.\ to $\COL_H$, in other words it is viewed as a \emph{row} Hilbert space.
\end{rem}

We also need a standard lemma on ``interchanging tensor products'': see \cite[Theorem 8.1.10]{ER_OSbook} for a proof.

\begin{lem}\label{l:tp-interchange}
Let $E$, $F$ and $G$ be operator spaces. There are complete contractions
\[ E\ptp (F\mintp G) \to (E\ptp F) \mintp G \quad,\quad (E\mintp F)\ptp G \to E \mintp (F\ptp G) , \]
both of which are the identity map when restricted to elementary tensors.
\end{lem}

\end{section}

\begin{section}{The main technical result}\label{s:mainresults}

We will deduce the main result (Theorem~\ref{t:mainthm}) from the special case $V=B(H)$, $W=B(K)$. It is convenient to reformulate this special case slightly, using Lemma~\ref{l:HS-trick}.

\begin{thm}\label{t:technical}
Let $H$ and $K$ be Hilbert spaces.
Then the linear map $\Phi_0: B(H) \tp B(K) \to B(S_2(K,H))$ that is defined by $\Phi_0(a\tp b)(c) = acb$
extends to a contractive linear map $\Phi_2: B(H)\ptp B(K) \to B(S_2(K,H))$.
\end{thm}

\begin{proof}
Write $S_0(K,H)$ for the space of finite-rank operators $K\to H$, $S_1(K,H)$ for the space of nuclear operators $K\to H$, and $S_\infty(K,H)$ for the space of all compact operators $K\to H$. Then we have three linear maps
\[ \Psi_j : B(H) \tp S_j(K,H) \tp B(K) \to S_j(K,H) \qquad(j=0,1,\infty), \]
each of which is defined on elementary tensors by $a\tp c\tp b \mapsto acb$.
The key point is that if we equip $S_1(K,H)$ and $S_\infty(K,H)$ with appropriate o.s.s., then we can extend both $\Psi_1$ and $\Psi_\infty$ to completely contractive linear maps on the threefold projective tensor product.

The details are as follows. Equip $H$ and $K$ with column o.s.s., and consider the two operator spaces $H\ptp K^*$ and $H\mintp K^*$.
The usual identification of $H\tp K^*$ with $S_0(K,H)$ extends to give two isometric isomorphisms of Banach spaces $H\ptp K^* \cong_1 S_1(K,H)$ and $H\mintp K^*\cong_1 S_\infty(K,H)$. (See e.g.~\cite{ER_OSbook}, Proposition~8.2.1.)

So we can identify $\Psi_0$, $\Psi_1$ and $\Psi_\infty$ with the following linear maps:
\[ \begin{array}{rccl}
\Psi_0: & B(H) \tp [ H\tp K^*] \tp B(K) & \to &H\tp K^*\\
\Psi_1: & B(H) \tp [ H\ptp K^*] \tp B(K) & \to &H\ptp K^*\\
\Psi_\infty :& B(H) \tp [ H\mintp K^*] \tp B(K) & \to & H\mintp K^*
\end{array}
\]
On elementary tensors, these maps satisfy
\[ \Psi_j(a\tp (v \tp \phi) \tp b) = av \tp b^\#\phi \qquad\text{($a\in B(H)$, $v\in H$, $\phi\in K^*$, $b\in B(K)$)} \]
for $j=0,1,\infty$. (Recall that $b^\# : K^* \to K^*$ is the adjoint in the Banach space sense.)

\eject
As noted in Remark~\ref{r:cc-actions}, there are completely contractive maps 
 $\lambda: B(H)\ptp H \to H$ and $\rho:K^*\ptp B(K)\to K$, satisfying
$\lambda(a\tp v)  =av$ ($a\in B(H)$, $v\in H$)
and
$\rho(\phi\tp b)  = b^\#\phi$ ($b\in B(K)$, $\phi\in K^*$).
Hence, by associativity of $\ptp$, we see that $\lambda \ptp \rho$ defines a complete contraction
\begin{equation}\label{eq:extend Psi_1}
 B(H) \ptp H\ptp K^*\ptp B(K) \to H\ptp K^*
\end{equation}
which extends $\Psi_1$. Furthermore, by using associativity of $\ptp$, and using Lemma~\ref{l:tp-interchange} twice, we have complete contractions
\[ \begin{aligned}
 B(H) \ptp [ H\mintp K^*] \ptp B(K) 
 &\cong_1 \left( B(H) \ptp [ H\mintp K^*] \right)\ptp B(K) \\
 & \to \left( [B(H) \ptp  H]\mintp K^* \right)\ptp B(K) \\
 & \to [ B(H) \ptp  H ] \mintp [K^* \ptp B(K) ].
\end{aligned} \]
Composing these maps with $\lambda\mintp \rho$, we obtain a complete contraction
\begin{equation}\label{eq:extend Psi_2}
 B(H) \ptp [ H\mintp K^*] \ptp B(K) \to H\mintp K^*
\end{equation}
which extends $\Psi_\infty$.

Fix $x\in B(H)\tp B(K)$ such that $\norm{x}_{B(H)\ptp B(K)}\leq 1$, and consider the corresponding linear map $\Phi_0(x): S_0(K,H)\to S_0(K,H)$.
 In view of \eqref{eq:extend Psi_1} and \eqref{eq:extend Psi_2}, we obtain contractive linear maps
\[ \Phi_1(x)  : S_1(K,H) \to S_1(K,H)\quad,\quad \Phi_\infty(x):S_\infty(K,H) \to S_\infty(K,H) , \]
which both extend the map $\Phi_0(x)$.

Viewing $(S_1(K,H), S_\infty(K,H))$ as a compatible interpolation couple of Banach spaces, we have $(S_1(K,H),S_\infty(K,H))_{1/2}\cong_1 S_2(K,H)$. By the Riesz--Thorin interpolation theorem, $\Phi_0(x)$ extends to a contractive linear map $\Phi_2(x): S_2(K,H)\to S_2(K,H)$.
It is now routine to check that $\Phi_2$ defines a linear contraction $B(H)\ptp B(K) \to B(S_2(K,H))$, which completes the proof of Theorem~\ref{t:technical}.
\end{proof}

\begin{proof}[Proof of Theorem~\ref{t:mainthm}]
Let $V$ and $W$ be operator spaces, and fix two completely isometric embeddings $j_V: V \hookrightarrow B(H)$ and $j_W:W\hookrightarrow B(K)$ for some choices of Hilbert spaces $H$ and~ $K$.
Consider the diagram \comment{Try to redo this without diagrams.sty?}
\[ \begin{diagram}[tight,width=7em,height=2.5em]
B(H) \ptp B(K) & \rTo^{\Psi_{B(H),B(K)}}  & B(H)\mintp \longopp{B(K)}  \\
 \uTo^{j_V\ptp j_W} & & \uTo_{j_V\mintp \til{j_W}} \\
V \ptp W  & \rDots & V\mintp \til{W}
\end{diagram} \]
where the top arrow $\Psi_{B(H),B(K)}$ restricts to the identity map on elementary tensors; note that $\Psi_{B(H),B(K)}$ is well-defined and contractive by combining Theorem~\ref{t:technical} with Lemma~\ref{l:HS-trick}.
Now observe that the left-hand vertical arrow in the diagram is a (complete) contraction, while the right-hand vertical arrow is a (complete) \emph{isometry} (using the ``injective'' property of $\mintp$). Hence, for any elementary tensor $x\in V\tp W$, we have
\[ \begin{aligned}
\norm{ x}_{V\mintp \til{W}}
& = \norm{(j_V\mintp \til{j_W})(x)}_{B(H)\mintp\longopp{B(K)}} \\
& \leq \norm{\Psi_{B(H),B(K)}}\norm{(j_V\ptp j_W(x)}_{B(H)\ptp B(K)} \leq \norm{x}_{V\ptp W}\;,
\end{aligned}
\]
which completes the proof.
\end{proof}


We finish by briefly justifying the claim, made after the statement of Theorem~\ref{t:mainthm}, that one cannot replace the projective tensor product with the Haagerup tensor product in that theorem. The result is probably well known but we include a proof for sake of completeness.

\begin{prop}\label{p:not-htp}
There exist $\Cst$-algebras $A$ and $B$ such that the identity map $A\tp B \to A\tp B$ does not extend to any bounded linear map $A\htp B \to A\mintp \til{B}$.
\end{prop}

\begin{proof}
For convenience we take $A=B(K)$ and $B=B(\Kbar)$ for an infinite-dimensional Hilbert space~$K$: it will be seen that we actually get separable counterexamples in the end.
Recall that $b\mapsto b^\top$ is a complete isometry
from $B(K)$ onto $\longopp{B(\Kbar)}$. Hence it suffices to show that the map ${\rm id}\tp\top : B(K)\tp B(\Kbar) \to B(K)\tp B(K)$ has no bounded extension to a map $B(K)\htp B(\Kbar) \to B(K)\mintp B(K)$.

Fix an infinite sequence of isometries $s_1,s_2,\dots$ in $B(K)$ with the property that the range projections $s_js_j^*$ are pairwise orthogonal. In particular, $s_j^*s_k=0$ for $j\neq k$ and $P_n\defeq\sum_{j=1}^n s_js_j^*$ is an orthogonal projection for each $n$.

 Let $x_n = \sum_{j=1}^n s_j\tp s_j^\top \in B(K) \tp B(\Kbar)$: then using the standard formula for the Haagerup tensor norm (see \cite[Chapter 5]{Pis_OSbook}),
\[ \norm{x_n}_{B(K)\htp B(\Kbar)} \leq \Norm{ \sum\nolimits_{j=1}^n s_js_j^* }^{1/2} \Norm{ \sum\nolimits_{k=1}^n (s_k^\top)^*s_k^\top }^{1/2} = \norm{P_n}^{1/2} \norm{P_n^\top}_{1/2} \leq 1 \;;\]
On the other hand, let $y_n = ({\rm id}\tp \top)(x_n) = \sum_{j=1}^n s_j \tp s_j \in B(K)\mintp B(K)$. By the $\Cst$-identity $\norm{y_n}  = \norm{y_n^*y_n}^{1/2}$, with both norms taken in $B(K)\mintp B(K)$. But since the $s_j$ are isometries with pairwise orthogonal ranges,
\[ y_n^*y_n = \sum\nolimits_{j,k=1}^n s_j^*s_k\tp s_j^*s_k = n I \tp I \;. \]
Hence $\norm{ ({\rm id}\tp \top)(x_n)}_{B(K)\mintp B(K)} \geq n^{1/2}  \norm{x_n}_{B(K)\htp B(\Kbar)}$, and since $n$ is arbitrary the result follows.
\end{proof}

\begin{rem}
Since $\htp$ and $\mintp$ are both injective tensor norms, the proof of Proposition~\ref{p:not-htp} shows that we could take $A$ to be the $\Cst$-algebra generated by the isometries $(s_j)_{j\geq 1}$, i.e.~the Cuntz algebra $\cO_\infty$, and $B$ to be the corresponding copy of $\cO_\infty$ inside $B(\Kbar)$. In particular, we get nuclear examples.
\end{rem}

\end{section}

\subsection*{Acknowledgements}
The author thanks Matthew Daws and Adam Skalski for various helpful exchanges on operator-space tensor products and conjugate Hilbert spaces, and for their comments and corrections on an earlier draft of this note.

He also wishes to acknowledge the admirable if idiosyncratic examples set over the years by John Rainwater (see  \cite{Rain-bio} for further details) who might, in times gone by, have taken credit for this note.


\newpage

\bibliography{twistostp}
\bibliographystyle{siam}

\bigskip
\noindent
Department of Mathematics and Statistics\\
Fylde College, Lancaster University\\
Lancaster, United Kingdom LA1 4YF
\hfill
Email: \texttt{y.choi1@lancaster.ac.uk}

\end{document}